\documentclass[a4paper,reqno,11pt]{amsart}

\usepackage{amsmath,amsthm,amssymb,mathtools,url,breqn}

\usepackage{amssymb}
\usepackage{amstext}
\usepackage{amsmath}
\usepackage{amscd}
\usepackage{amsthm}
\usepackage{amsfonts}
\usepackage{enumerate}
\usepackage{graphicx}
\usepackage{color}
\usepackage{here} 
\usepackage{bm} 
\usepackage{mathrsfs}
\usepackage{mathtools}
\usepackage{latexsym}
\usepackage{booktabs}
\usepackage{fancyhdr}
\usepackage{subcaption}
\usepackage{tikz}
\usepackage{tikz-cd}
\usetikzlibrary{arrows}
\usetikzlibrary{decorations.markings}
\usetikzlibrary{patterns,decorations.pathreplacing}

\makeatletter
  
\@addtoreset{equation}{section}
\makeatother

\newcommand{\ZZ}{\mathbb{Z}}
\newcommand{\QQ}{\mathbb{Q}}
\newcommand{\RR}{\mathbb{R}}

\newcommand{\kk}{\Bbbk}

\newcommand{\xb}{\mathbf{x}}

\newcommand{\fkg}{\mathfrak{g}}

\newcommand{\Hom}{\operatorname{Hom}}

\def\opn#1#2{\def#1{\operatorname{#2}}} 
\opn\Cl{Cl} \opn\conv{conv} \opn\deg{deg} \opn\rank{rank} \opn\Spec{Spec} \opn\Stab{Stab}\opn\cone{cone} \opn\End{End} \opn\Hom{Hom} \opn\mod{mod} \opn\gldim{gldim} \opn\pdim{pdim}\opn\Block{Block} \opn\Pyr{Pyr}
\opn\gin{gin}\opn\inn{in} \opn\cok{coker} \opn\core{core}
\opn\pd{pd} \opn\Soc{Soc} \opn\Ap{Ap}
\opn\PF{PF} \opn\t{t} \opn\F{F} \opn\e{e}
\opn\m{m} \opn\G{G} \opn\g{g} \opn\H{H}
\opn \embdim{embdim} \opn\var{var}
\opn{\mult}{mult} \opn{\emb}{emb}
\opn{\gap}{Gap}

\newtheorem{thm}{Theorem}[section]
\newtheorem{cor}[thm]{Corollary}
\newtheorem{lem}[thm]{Lemma}
\newtheorem{prop}[thm]{Proposition}

\theoremstyle{definition}
\newtheorem{defi}[thm]{Definition}
\newtheorem{ex}[thm]{Example}

\theoremstyle{remark}
\newtheorem{rem}[thm]{Remark}

\begin{document}

\title{Conditions of multiplicity and applications \\for almost Gorenstein graded rings}
\author{Koji Matsushita and Sora Miyashita}

\address[K. Matsushita]{Department of Pure and Appl. Mathematics, Graduate School of Information Science and Technology, Osaka University, Suita, Osaka 565-0871, Japan}
\email{k-matsushita@ist.osaka-u.ac.jp}
\address[S. Miyashita]{Department of Pure And Appl. Mathematics, Graduate School Of Information Science And Technology, Osaka University, Suita, Osaka 565-0871, Japan}
\email{u804642k@ecs.osaka-u.ac.jp}

\subjclass[2020]{
Primary 13H10; Secondary 13M05} 
\keywords{almost Gorenstein, multiplicity, tensor products, edge rings,
stable set rings,
numerical semigroup rings}

\maketitle

\begin{abstract}
In this paper, 
we prove that if Cohen--Macaulay local/graded rings $A$, $B$ and $R$ satisfy certain conditions regarding multiplicity and Cohen--Macaulay type, then almost Gorenstein property of $R$ implies Gorenstein properties for all of $A$, $B$ and $R$.
We apply our theorem to tensor products of semi-standard graded rings
and some classes of affine semigroup rings,
i.e.,
numerical semigroup rings, edge rings and stable set rings.
\end{abstract}

\bigskip

\section{Introduction}
An almost Gorenstein ring, which we focus on in this paper, is one of
generalized notion of Gorenstein rings.
This notion was initially introduced by Barucci and Fr{\"o}berg \cite{barucci1997one} in the case where the local rings are of dimension one and analytically unramified.
Their work inspired Goto, Matsuoka and Phuong \cite{goto2013almost} to extend the notion of almost Gorenstein local rings for any one-dimensional Cohen--Macaulay rings.
Moreover, Goto, Takahashi and Taniguchi \cite{goto2015almost} defined the almost Gorenstein local rings for any dimension.
After that, the $h$-vectors of almost Gorenstein standard graded rings are studied by Higashitani in \cite{higashitani2016almost}.
The almost Gorenstein property of affine semigroup rings such as numerical semigroups \cite{herzog2019almost,nari2013symmetries}, edge rings \cite{bhaskara2023h,higashitani2022levelness, higashitani2023h} and stable set rings \cite{miyazaki2023non} has been studied.

Let $(R, \m)$ be a Cohen--Macaulay local ring.
Then $R$ is said to be an {\it almost Gorenstein local ring} if $R$ admits a canonical module $\omega_R$ and
\begin{align}\nonumber
\text{there exists an $R$-monomorphism $\phi: R \hookrightarrow \omega_R$}
\end{align}
such that $C:=\cok(\phi)$ is an Ulrich $R$-module, i.e., $\mu(C) = e(C)$ (\cite[Definition 3.3]{goto2015almost}). Here, $\mu(C)$ (resp. $e(C)$) denotes the number of elements in a minimal generating system of $C$ (resp. the multiplicity of $C$ with respect to $\m$).
The zero module is regarded as an Ulrich $R$-module.
Similarly to local rings, a Cohen--Macaulay graded ring $R=\bigoplus_{n \ge 0}R_n$ with $\kk=R_0$ a field is called an {\it almost Gorenstein graded ring} if $R$ admits a graded canonical module $\omega_R$ and
\begin{align}\nonumber
\text{there exists an $R$-monomorphism $\phi: R \hookrightarrow \omega_R(-{a_R})$ of degree 0}
\end{align}
such that $C:=\cok(\phi)$ is an Ulrich $R$-module
(\cite[Definition 8.1]{goto2015almost}).
Here,
$a_R=-\min \{i:(\omega_R)_i \neq 0\}$ is the $a$-invariant of $R$ and
$\omega_R(-{a_R})$ is the graded $R$-module defined as having the same underlying $R$-module structure as $\omega_R$ with the grading determined by $[\omega_R(-{a_R})]_n = [\omega_R]_{n-{a_R}}$ for all $n \in \Bbb Z$.
In this paper, we introduce the following.
\begin{defi}
Let $(R,\m)$ be a Cohen--Macaulay local ring that admits a canonical module
$\omega_R$.
Assume that there exists an $R$-monomorphism
$\phi : R \hookrightarrow \omega_R$
and
we define $\delta_{\phi}(R)$ as $e(\cok(\phi))$.
When there is no risk of confusion about the monomorphism we simply write $\delta_{\phi}(R)$ as $\delta(R)$ without explicitly mentioning $\phi$.
Moteover, we call $R$ is {\it almost Gorenstein with respect to $\phi$} if $\cok(\phi)$ is an Ulrich $R$-module.
We call {\it $R$ satisfies  $(*)$}
if there exists an $R$-monomorphism
$\phi : R \hookrightarrow \omega_R$
such that $\phi(1_R) \notin \m \omega_R$.
\end{defi}

If $R$ is non-regular,
then condition $(*)$ is a necessary condition for $R$ to be almost Gorenstein (see \cite[Corollary 3.10]{goto2015almost}).
We introduce the graded versions of above.

\begin{defi}
Let $R=\bigoplus_{i \geq 0}R_i$ be a Cohen--Macaulay graded ring.
We call {\it$R$ satisfies $(*)$}
if there exists an $R$-monomorphism $\phi : R \hookrightarrow \omega_R(-a_R)$ of degree 0.
Note that $\phi$ satisfies $\phi(1_R)
\notin \m \omega_R$
where $\m=\bigoplus_{n > 0}R_n$.
If $R$ satisfies $(*)$,
take an $R$-monomorphism $\phi : R \hookrightarrow \omega_R(-a_R)$ of degree 0 and
we define $\delta_{\phi}(R)$ as $e(\cok(\phi))$.
When there is no risk of confusion about the monomorphism
we simply write $\delta_{\phi}(R)$ as $\delta(R)$.
Moreover, we call $R$ is {\it almost Gorenstein with respect to $\phi$} if $\cok(\phi)$ is an Ulrich $R$-module.
\end{defi}

Now let us explain our main theorem.
Let $A$ and $B$ be positively graded $\kk$-algebras and let $R$ be a $\kk$-algebra constructed
from $A$ and $B$
in a certain way.
In this case, some ring-theoretic properties of $R$ may be inherited by $A$ and $B$.
For example, let $R = A \otimes_\kk B$ be the tensor product of $A$ and $B$, then it is known that $A$ and $B$ are Gorenstein if and only if $R$ is Gorenstein.
Actually, a similar assertion holds for the tensor product of almost Gorenstein semi-standard graded rings
and the ``gluing" of numerical semigroups defined by \cite{De, JCR},
which is a
distinct operation from tensor product.

In this paper, a more general statement is proved that extends the cases of the tensor product of semi-standard graded rings and the gluing of numerical semigroups.
The following are our main results:
Here, $r(R)$ denotes the Cohen--Macaulay type of $R$.
\begin{thm}[{see Theorem~\ref{MMMMMlocal}}]
Let $A$, $B$ and $R$ be Cohen--Macaulay local rings.
Assume that $A$ and $B$ satisfy $(*)$ 
and there exists an $R$-monomorphism $\phi:R\hookrightarrow \omega_R$.
Moreover, we assume $e(A) > 1$, $e(B)>1$ and
\begin{align}\label{INEQUL5}
\delta_\phi(R) \geq e(B)\delta({A})+e(A)\delta({B}).
\end{align}
Then the following conditions are equivalent:
\begin{itemize}
\item[(1)] $R$ is almost Gorenstein with respect to $\phi$ and $r(R) \leq r(A)r(B) $;
\item[(2)] $A$ and $B$ Gorenstein
and the equality of (\ref{INEQUL5}) holds.
\end{itemize}
If this is the case, $R$ is Gorenstein.
\end{thm}
If we consider graded case, we have the following.
\begin{thm}[{see Theorem~\ref{MMMMM}}]
Let $A$, $B$ and $R$ be Cohen--Macaulay positively graded rings.
Assume that $A$, $B$ and $R$ satisfy $(*)$, $e(A) > 1$, $e(B)>1$ and
\begin{align}\label{INEQUL}
\delta_\phi(R) \geq e(B)\delta({A})+e(A)\delta({B}).
\end{align}
Then the following conditions are equivalent:
\begin{itemize}
\item[(1)] $R$ is almost Gorenstein with respect to $\phi$ and $r(R) \leq r(A)r(B) $;
\item[(2)] $R$ is Gorenstein;
\item[(3)] $A$ and $B$ Gorenstein
and the equality of (\ref{INEQUL}) holds.
\end{itemize}
\end{thm}
When does the inequality (\ref{INEQUL}) hold?
Actually, the tensor product of semi-standard graded rings or the gluing of numerical semigroups satisfies the inequality (\ref{INEQUL}).
So we get the following results about the tensor products (or the quotient rings divided by
its regular sequences)
and the gluing of numerical semigroups. The second one is a well-known result (see \cite[Theorem 6.7]{nari2013symmetries}
and \cite{De}).

\begin{cor}[{see Corollary~\ref{cor_tensor}}]
Let $A$ and $B$ be semi-standard graded Cohen--Macaulay $\kk$-algebras over a field $\kk$
such that $A$ and $B$ satisfy $(*)$ and $e(A) > 1$, $e(B)>1$.
In addition, let $T=A\otimes_\kk B$ and
let $\xb=x_1,\ldots,x_n$
be a homogeneous regular sequence on $T$
 with $\deg(x_k)=a_k$ for $1 \le k \le n$
such that $T/(\xb)$
satisfies $(*)$.
Set $R=T$ or $R=T/(\xb)$,
then the following conditions are equivalent:
\begin{itemize}
\item[(1)] $R$ is almost Gorenstein;
\item[(2)] $R$ is Gorenstein;
\item[(3)] $A$ and $B$ are Gorenstein.
\end{itemize}
\end{cor}

\begin{cor}[{see Corollary~\ref{NUMEGLUINGOYO}}]
Let $H_1$ and $H_2$ be numerical semigroups
and let $H$ be the gluing of $H_1$ and $H_2$.
Then the following conditions are equivalent:
\begin{itemize}
\item[(1)] $\kk[H]$ is almost Gorenstein;
\item[(2)] $\kk[H]$ is Gorenstein;
\item[(3)] $\kk[H_1]$ and $\kk[H_2]$ are Gorenstein.
\end{itemize}
\end{cor}
\noindent
Here, $\kk[H]$ denotes the numerical semigroup ring of $H$.
Please refer to subsection \ref{NUMENUME} for the definition of the gluing of numerical semigroups.

We also provide applications of our results to two toric rings arising from graphs; edge rings and stable set rings.

\begin{thm}[{see Theorem~\ref{thm_ed}}]\label{main_edge}
Let $G_1$ and $G_2$ be simple graphs.
Suppose that $G_1$ is bipartite and $\kk[G_2]$ is Cohen--Macaulay with $e(\kk[G_i])>1$ for $i=1,2$.
Then the following are equivalent: 
\begin{itemize}
\item[(1)] $\kk[G_1\sharp G_2]$ is almost Gorenstein;
\item[(2)] $\kk[G_1\sharp G_2]$ is Gorenstein;
\item[(3)] $\kk[G_1]$ and $\kk[G_2]$ are Gorenstein.
\end{itemize}
\end{thm}
\noindent Here, $\kk[G]$ denotes the edge ring of a graph $G$ and $G_1\sharp G_2$ denotes the clique sum of $G_1$ and $G_2$.

\begin{thm}[{see Theorem~\ref{thm_st}}]\label{main_stab}
Let $G_1$ and $G_2$ be simple graphs.
Suppose that $\Stab_\kk(G_i)$ is Cohen--Macaulay with $e(\Stab_\kk(G_i))>1$ for $i=1,2$.
Then the following are equivalent:
\begin{itemize}
\item[(1)] $\Stab_\kk(G_1 + G_2)$ is almost Gorenstein;
\item[(2)] $\Stab_\kk(G_1 + G_2)$ is Gorenstein;
\item[(3)] $\Stab_\kk(G_1)$ and $\Stab_\kk(G_2)$ are Gorenstein.
\end{itemize}
\end{thm}
\noindent Here, $\Stab_\kk(G)$ denotes the stable set ring of a graph $G$ and $G_1+G_2$ denotes the join of $G_1$ and $G_2$.

The structure of this paper is as follows.
In Section 2, we prepare some facts and definitions for the discussions later; we collect basic definitions and properties of commutative rings, numerical semigroups and toric rings arising from graphs.
In Section 3, we show our main theorem.
We also discuss conditions appearing in our main theorem.
Moreover, we apply our theorem to tensor products
of semi-standard graded rings.
In Section 4, we apply our main theorem to some classes of affine semigroup rings,
i.e.,
numerical semigroup rings, edge rings and stable set rings.

\subsection*{Acknowledgement}
We are grateful to Professor Akihiro Higashitani for his
comments and instructive discussions.
The first author is partially supported by
Grant-in-Aid for JSPS Fellows Grant JP22J20033.
\section{Preliminaries}
The goal of this section is to prepare the required materials for the discussions of our main results.
\subsection{Preliminaries on commutative rings}

In this subsection, we recall some fundamental materials (consult, e.g., \cite{bruns1998cohen} for the introduction to Cohen--Macaulay rings). 
\begin{itemize}
\item Let $\omega_R$ denote a canonical module of $R$. Let $a_R$ denote the $a$-invariant of $R$, i.e., $a_R=-\min\{n:(\omega_R)_n \neq 0\}$. 
\item For an $R$-module $M$, we use the following notation:
\begin{itemize}
\item Let $\mu(M)$ denote the number of a minimal generating system of $M$ as an $R$-module.
\item Let $e(M)$ denote the multiplicity of $M$. \item Note that the inequality $\mu(M) \leq e(M)$ always holds.
\end{itemize}
\item Let $r(R)$ denote the Cohen--Macaulay type of $R$. Note that $r(R)=\mu(\omega_R)$.
We see that Cohen--Macaulay ring $R$ is Gorenstein if and only if $r(R)=1$.
\item For a graded $R$-module $M$, we denote $\H(M,t)$ the Hilbert series of $M$, i.e., 
$$\H(M,t)=\sum_{n \in \ZZ}\dim_\kk M_n t^n,$$ where $\dim_\kk$ stands for the dimension as a $\kk$-vector space.
\item
Let $R=\bigoplus_{i\geq 0}R_i$ be a graded ring.
$R$ is called {\it semi-standard graded} if $R$ is finitely generated $\kk[R_1]$-module.
Moreover, $R$ is called {\it standard graded} if $R=\kk[R_1]$.
If $R$ is a semi-standard graded ring, its Hilbert series is of the form
$$\H(R,t) =\frac{h_0+h_1 t + \cdots +h_st^s}{(1-t)^{\dim R}}$$ 
where $h_0+h_1 t + \cdots +h_st^s
\in \ZZ[t]$.
We call this polynomial the {\it $h$-polynomial} of $R$ and denote it as $h_R(t)$.
Note that $\sum_{i=0}^s h_i \ne 0$ and $h_s \ne 0$.
We call the integer sequence $(h_0,h_1,\ldots,h_s)$ the {\it $h$-vector} of $R$.
We always have $h_0=1$.
If a semi-standard graded ring $R$ is Cohen--Macaulay,  its $h$-vector satisfies
$h_i \geq 0$ for any $i$. Moreover, if $R$ is standard graded, we have $h_i > 0$ for any $i$. 
For further information on the $h$-vectors of Cohen--Macaulay semi-standard (resp. standard) graded rings, see \cite{stanley1991hilbert} (resp. \cite{stanley2007combinatorics}).
Note that $h_s \leq r(R)$ holds. Moreover, we see that $d+a_R=s$ and $e(R)=h_0+h_1+\cdots +h_s$ (see \cite[Section 4.4]{bruns1998cohen}).
\end{itemize}

Let $R$ be a local (resp. positively graded) Cohen--Macaulay ring. Assume that $R$ satisfies $(*)$
and take a (graded) monomorphism $\phi: R \hookrightarrow \omega_R$ with $1_R \notin \m \omega_R$.
Put $C=\cok(\phi)$. Then $C$ is a Cohen--Macaulay $R$-module of dimension $d-1$ 
if $C\not=0$ (see \cite[Lemma 3.1]{goto2015almost}).
Note that $C=0$ if and only if $R$ is Gorenstein.
Moreover, there are some relation between $\mu(C)$ and $r(R)$.
We can check the following by the same proof as \cite[Proposition 2.2]{higashitani2016almost}.

\begin{prop}\label{gradedmuformula}
Let $R$ be a Cohen--Macaulay local (resp. positively graded) ring which satisfies $(*)$.
Then we have $\mu(C)=r(R)-1$.
\end{prop}



Let $R$ be a semi-standard graded ring, then the following result is true.
This is a semi-standard graded version of \cite[Proposition 2.4]{higashitani2016almost}.

\begin{prop}[{see \cite[Theorem 4.5]{miyashita2023comparing}}]\label{hformula}
Let $R$ be a Cohen--Macaulay semi-standard graded ring which satisfies $(*)$ and let $(h_0,\ldots,h_s)$ be its $h$-vector.
Take any monomorphism $\phi:R \hookrightarrow \omega_{R}(-a_R)$ with degree 0.
Then the Hilbert series of $C:=\cok(\phi)$ is
\begin{align}\label{ccc}
\H(C,t)= \frac{\sum_{j=0}^{s-1}((h_s+\cdots+h_{s-j})-(h_0+\cdots+h_j))t^j}{(1-t)^{\dim R-1}}.
\end{align}
In particular, $\delta_\phi(R)=\sum_{j=0}^{s-1}((h_s+\cdots+h_{s-j})-(h_0+\cdots+h_j))=\sum_{j=0}^{s}(2j-s)h_j$.
\end{prop}

The following is well-known, but we also provide a proof.

\begin{prop}\label{prop_type<e}
Let $R$ be a Cohen--Macaulay local (or positively graded) ring.
If $e(R) \geq 2$, then we have $r(R)<e(R)$.
\end{prop}
\begin{proof}
Put $d=\dim R$.
Since $R$ is Cohen--Macaulay,
there exists a regular sequence
$\xb$ on $\m$
such that $e(R)=l(R/\xb)$.
Thus we may assume $d=0$
since $r(R)=r(R/(\xb))$.
By the assumption,
we have
$\Soc(R) \neq R$ so $\Soc(R) \subset \m$,
where $\Soc(R)=\{x \in R : x\m=0 \}$.
Therefore, we have
$e(R)=\dim_\kk \Soc(R) \leq \dim_\kk \m<\dim_{\kk}R=e(R)$ so $r(R)<e(R)$.
\end{proof}

The following is a formula for calculating the Hilbert series of the canonical module.

\begin{prop}[{\cite[Corollary 4.4.6 (a)]{bruns1998cohen}}]\label{canonicalHilbert}
Let $\kk$ be a field and let $R$ be a $d$-dimensional Cohen--Macaulay positively graded $\kk$-algebra.
Then
$\H(\omega_{R}(-a_R),t)=(-1)^dt^{a_R}\H(R,t^{-1})$.
\end{prop}

\subsection{Numerical semigroups}\label{NUMENUME}
In this subsection, we recall the definitions and properties related to numerical semigroups.
See, e.g., \cite{rosales2009numerical} for the introduction to numerical semigroup.
A \textit{numerical semigroup} $H$ is a subset of $\mathbb{N}$ which is closed under addition,
 contains the zero element and whose complement in $\mathbb{N}$ is finite.
 We denote $\kk[H]$ as the numerical semigroup ring
 of $H$
 Every numerical semigroup $H$ admits a finite generating system, that is, there exist $a_1,\ldots, a_n\in H$ with 
 $H=\langle a_1,\ldots, a_n\rangle =\{\lambda_1a_1+\cdot \cdot \cdot +\lambda _na_n\mid \lambda_1,\ldots, \lambda _n\in \mathbb{N}\}$. \par
 Let $H$ be a numerical semigroup.
 We denote $G(H)=\{a_1<a_2<\cdot \cdot \cdot <a_n\}$ as the minimal generating system of $H$. 
 We call $a_1$ {\it the multiplicity} of $H$ and denote it by $\mult(H)$, and we call $n$ 
 {\it the embedding dimension} of $H$ and denote it by $\emb(H)$. In general, $\emb(H)\leq \mult(H)$.
The finite set $\mathbb{N}\setminus H$ is called {\it the set of gaps} of $H$.
 If $H$ is a numerical semigroup, the largest integer in $\mathbb{N}\setminus H$ is called 
 {\it Frobenius number} of $H$ and we denote it by $\F(H)$.
 We say that $H$ is {\it symmetric} if for every $z \in \mathbb{Z}$, either $z\in H$ or $\F(H)-z \in H$. 
 We say that an integer $x$ is a {\it pseudo-Frobenius number} of $H$ if 
$x \not \in H$ and $x+h \in H$ for all $h \in H, h\ne 0$.
 We denote the set of pseudo-Frobenius numbers of $H$ by $\PF(H)$.
 The cardinality in $\PF(H)$ is called the {\it type} of $H$, denoted by $\t(H)$.
 It is known that
$\mult(H)=e(\kk[H])$, $\t(H)=r(\kk[H])$ (see \cite{goto1978graded}).
Moreover, $H$ is symmetric if and only if $\kk[H]$ is Gorenstein (see \cite{HerzogKunz}).
Let $H$ be a numerical semigroup and let $s \in H \setminus \{0\}$;
the \textit{Ap\'{e}ry} set with respect to a nonzero $s \in H$ is defined as
$$\Ap(H,s) = \{a \in H : a-s \notin H\}.$$
Note that $|\Ap(H,s)|=s$ and there is a relation between Ap\'{e}ry sets and Hilbert series. 

\begin{prop}[{see \cite{ramirez2009numerical}}]\label{aperyHilbert}
Let $H$ be a numerical semigroup and let $s \in H \setminus \{0\}$;
$$\H(\kk[H],t)= \frac{1}{1-t^s}\sum_{a \in \Ap(H,s)} t^a$$
\end{prop}

 The concept of the gluing of numerical semigroups was defined in \cite{De,JCR}.
 
 \begin{defi}[{see \cite{De,JCR}}]\label{gluing} Let $H_1=\langle a_1, a_2, \ldots, a_n \rangle$
 and $H_2=\langle b_1, b_2, \ldots, b_m \rangle$ be numerical semigroups.
 Take $y \in H_1 \setminus G(H_1)$ and $x \in H_2 \setminus G(H_2)$
 such that $\gcd(x, y)=1$.
 We call
 \[H=\langle xH_1, yH_2 \rangle = \langle xa_1, xa_2, \ldots, xa_n, yb_1, yb_2, \ldots, yb_m \rangle \]
 the {\it gluing} of $H_1$ and $H_2$ (with respect to $x$ and $y$).
 \end{defi}

The following is well-known result about the gluing of numerical semigroups.
 
 \begin{thm}[{\cite[Theorem 6.3]{nari2013symmetries}}
 and \cite{De}]\label{NUMEE}
 Let $H_1$ and $H_2$ be numerical semigroups 
and let $H$ be the gluing of $H_1$ and $H_2$.
Then the following are true:
\begin{itemize}
\item[(1)] If $\kk[H_1]$ and $\kk[H_2]$ are Gorenstein,
then $\kk[H]$ is Gorenstein;
\item[(2)] If $\kk[H]$ is almost Gorenstein,
 then $\kk[H_1]$ and $\kk[H_2]$ are Gorenstein.
 \end{itemize}
 \end{thm}

 Now, let $H=\langle xH_1, yH_2 \rangle$ be the gluing of $H_1$ and $H_2$ with $\gcd(x, y)=1$ and $xy \in xH_1 \cap yH_2$.
 Nari (\cite{nari2013symmetries}) proved the following
 to show Theorem \ref{NUMEE} (2).
 
 \begin{lem}[{see \cite[Lemma 6.5 and Proposition 6.6]{nari2013symmetries}}]\label{Nari1}
 Let $H=\langle xH_1, yH_2 \rangle$ be the gluing of $H_1$ and $H_2$, then we have
 \[\Ap(H, xy)=\{xs + yt \mid s \in \Ap(H_1, y), t \in \Ap(H_2, x) \}, \]
 %
 \[\PF(H)=\{ xf + yf' + xy \mid f \in \PF(H_1), f'\in \PF(H_2) \}\]
 with $\t(H)=\t(H_1)\t(H_2)$.
 In particular, $F(H)=x\F(H_1)+y\F(H_2)+xy$.
 \end{lem}
   
\subsection{Toric rings arising from graphs}
In this subsection, we prepare some notions and notation on graphs and define edge rings and stable set rings, which are toric rings arising from graphs.
For the fundamental materials on graph theory and toric rings, consult, e.g., \cite{diestel2017graph} and \cite{herzog2018binomial}, respectively.

For a positive integer $d$, let $[d]:=\{1,\ldots,d\}$. 
Consider a finite simple graph $G$ on the vertex set $V(G)=[d]$ with the edge set $E(G)$.
We say that $S \subset V(G)$ is a \textit{stable set} or an \textit{independent set} (resp. a \textit{clique}) 
if $\{v,w\} \not\in E(G)$ (resp. $\{v,w\} \in E(G)$) for any distinct vertices $v,w \in S$. 
Note that the empty set and each singleton are regarded as stable sets. 

We introduce two ways to construct a new graph from two given graphs $G_1$ and $G_2$:
\begin{itemize}
\item Suppose that $V(G_1)\cap V(G_2)$ is a clique of both $G_1$ and $G_2$.
Then the \textit{clique sum} $G_1\sharp G_2$ of $G_1$ and $G_2$ along $V(G_1)\cap V(G_2)$ is the graph on the vertex set $V(G_1)\cup V(G_2)$ with the edge set $E(G_1)\cup E(G_2)$.
\item Suppose that $V(G_1)\cap V(G_2)=\emptyset$.
Then the \textit{join} $G_1+G_2$ of $G_1$ and $G_2$ is the graph on the vertex set $V(G_1)\cup V(G_2)$ with the edge set $E(G_1)\cup E(G_2)\cup \{\{i,j\} : i \in V(G_1), j\in V(G_2)\}$.
\end{itemize}

\medskip

Now, we define our toric rings:
\begin{defi}
Consider the morphism $\phi_G$ of $\kk$-algebras:
$$\phi_G : \kk[x_{\{i,j\}} : \{i,j\}\in E(G)] \to \kk[t_1,\ldots,t_d], \text{ induced by } \phi_G(x_{\{i,j\}})=t_it_j.$$
Then we denote the image (resp. the kernel) of $\phi_G$ by $\kk[G]$ (resp. $I_G$) and call $\kk[G]$ the \textit{edge ring} of $G$. 
\end{defi}
\begin{defi}
Consider the morphism $\psi_G$ of $\kk$-algebras:
$$\psi_G : \kk[x_S : S \text{ is a stable set of }G] \to \kk[t_1,\ldots,t_{d+1}], \text{ induced by } \psi_G(x_S)=\left(\prod_{i \in S}t_i\right)t_{d+1}.$$
Then we denote the image (resp. the kernel) of $\psi_G$ by $\Stab_\kk(G)$ (resp. $J_G$) and call $\Stab_\kk(G)$ the \textit{stable set ring} of $G$.
\end{defi}
The edge ring (resp. stable set ring) of $G$ is regarded as a standard graded $\kk$-algebra by setting $\deg(\phi_G(x_{\{i,j\}}))=1$ (resp. $\deg(\psi_G(x_S))=1$).
Moreover, we can regard $I_{G_i}$ (resp. $J_{G_i}$) as a subset of $I_G$ (resp. $J_G$) since each edge (resp. stable set) of $G_1$ or $G_2$ is that of $G_1\sharp G_2$ (resp. $G_1+G_2$).

\section{Conditions for almost Gorenstein rings}

In this section, we show our main theorem and discuss conditions appearing in our main theorem.
Moreover, we apply our theorem to tensor products
of semi-standard graded rings and the quotient rings divided by
its regular sequences.

\subsection{Proof of main theorem}
The following is a key lemma.

\begin{lem}\label{mysterious}
Let $A$ and $B$ be Cohen--Macaulay local
(or graded) rings with $e(A) > 1$ and $e(B)>1$.
If there exist $\gamma_1, \gamma_2 \in \RR_{\geq 0}$ such that
$r(A)\leq \gamma_1+1$, $r(B)\leq \gamma_2+1$
and
\begin{align}\label{HUTOUSIKII}
e(B)\gamma_1+e(A)\gamma_2 \leq 
 r(A)r(B)-1,
\end{align}
then $A$ and $B$ are Gorenstein and $\gamma_i=0$ for $i=1,2$.

\begin{proof}
We get the following inequality by the assumptions:
\begin{equation}\label{HUTOUSIKI}
e(B)\gamma_1+e(A)\gamma_2 \leq r(A)r(B)-1 \leq \gamma_1\gamma_2+\gamma_1+\gamma_2.
\end{equation}
Thus we have
\begin{equation*}
X:=(e(A)-1-\gamma_1)(e(B)-1-\gamma_2)-(e(A)-1)(e(B)-1) \geq 0.
\end{equation*}
It follows from $(e(A)-1)(e(B)-1)>0$ and $X\geq 0$ that $(e(A)-1-\gamma_1)(e(B)-1-\gamma_2)>0$, equivalently, ``$e(A)-1-\gamma_1> 0$ and $e(B)-1-\gamma_2> 0$'' or ``$e(A)-1-\gamma_1< 0$ and $e(B)-1-\gamma_2< 0$''.
Suppose that $e(A)-1-\gamma_1< 0$ and $e(B)-1-\gamma_2< 0$.
This hypothesis and Proposition~\ref{prop_type<e} imply that
\begin{equation}\notag
\begin{split}
e(B)\gamma_1+e(A)\gamma_2
&> e(B)(e(A)-1)+e(A)(e(B)-1)\\
&\geq r(B)(r(A)-1)+r(A)(r(B)-1)\\
&\geq r(A)r(B)-1,
\end{split}
\end{equation}
a contradiction to (\ref{HUTOUSIKII}).

Thus we have $e(A)-1-\gamma_1> 0$ and $e(B)-1-\gamma_2> 0$.
On the other hand, in this situation, we can see that $X<0$ unless $\gamma_1=\gamma_2=0$.
Then we obtain $\gamma_1=\gamma_2=0$, and hence $r(A)=r(B)=1$ by our assumption. Therefore, $A$ and $B$ are Gorenstein, as desired.
\end{proof}
\end{lem}

\begin{thm}\label{MMMMMlocal}
Let $A$, $B$ and $R$ be Cohen--Macaulay local rings.
Assume that $A$ and $B$ satisfy  $(*)$ 
and there exists an $R$-monomorphism $\phi:R\hookrightarrow \omega_R$.
Moreover, we assume $e(A)>1$, $e(B)>1$ and
\begin{align}\label{FFRT}
\delta_\phi(R) \geq e(B)\delta({A})+e(A)\delta({B}).
\end{align}
Then the following conditions are equivalent:
\begin{itemize}
\item[(1)] $R$ is almost Gorenstein with respect to $\phi$ and $r(R) \leq r(A)r(B) $;
\item[(2)] $A$ and $B$ Gorenstein
and the equality of (\ref{FFRT}) holds.
\end{itemize}
If this is the case, $R$ is Gorenstein.
\begin{proof}
It is clear that $(2) \Rightarrow (1)$ since $R$ is Gorenstein if and only if $\delta_\phi({R})=0$.
We show (1) implies (2).
If $R$ is regular, then $R$ is Gorenstein so the assertion follows from $\delta_\phi(R)=0$. So we may assume that $R$ is not regular.
Since $R$ is non-regular almost Gorenstein, we have $\delta_\phi({R})=r(R)-1$ by \cite[Corollary 3.10]{goto2015almost} (if we consider the graded case, the equality holds without this corollary).
Then we have
\begin{equation}\label{HUTOUSIKI}
e(B)\delta({A})+e(A)\delta({B}) \leq \delta_\phi(R)=r(R)-1\leq r(A)r(B)-1.
\end{equation}
Therefore,
since $r(A)-1\leq \delta({A})$ and $r(B)-1\leq \delta(B)$ by Proposition~\ref{gradedmuformula}
and
$e(B)\delta({A})+e(A)\delta({B}) \leq r(A)r(B)-1$,
$A$ and $B$ are Gorenstein and $\delta(A)=\delta(B)=0$ by Lemma~\ref{mysterious}.
Then we have $e(B)\delta({A})+e(A)\delta({B})=\delta_\phi({R})=0$
by (\ref{HUTOUSIKI}), as desired.
\end{proof}
\end{thm}

\begin{thm}\label{MMMMM}
Let $A$, $B$ and $R$ be Cohen--Macaulay positively graded rings.
Assume that $A$, $B$ and $R$ satisfy $(*)$. Moreover, we assume $e(A)>1$, $e(B)>1$ and
\begin{align}\label{multiplicitycondition1}
\delta_\phi(R) \geq e(B)\delta({A})+e(A)\delta({B}).
\end{align}
Then the following conditions are equivalent:
\begin{itemize}
\item[(1)] $R$ is almost Gorenstein with respect to $\phi$ and $r(R) \leq r(A)r(B) $;
\item[(2)] $R$ is Gorenstein;
\item[(3)] $A$ and $B$ Gorenstein
and the equality of (\ref{multiplicitycondition1}) holds.
\end{itemize}
\begin{proof}
We can show $(1) \Rightarrow (3) \Rightarrow (2)$
by the same proof as Theorem \ref{MMMMMlocal}.
We show that $(2) \Rightarrow (1)$.
It is enough to show that $\cok(\phi)=0$.
Since $R$ is Gorenstein,
there exists a graded isomorphism $\alpha: \omega_R(-a_R) \rightarrow R$.
Put $\psi=\phi \circ \alpha$.
Then $\psi$ is an isomorphism since $(\omega_R(-a_R))_i$ is finite dimensional vector space for any $i \in \ZZ$. Thus we get $\cok(\phi)=0$.
\end{proof}
\end{thm}

\begin{rem}\label{independence}
If we add the assumption that $R$ is semi-standard graded to Theorem \ref{MMMMM},
$(1)$ does not depend on how $\phi$ is chosen,
that is, $(1)$ if and only if $R$ is almost Gorenstein and $r(R) \leq r(A)r(B)$ because of
Proposition \ref{gradedmuformula} and \ref{hformula}.
\end{rem}

\begin{rem}\label{drop_assumption}
(a) The essential part of the proof of the above theorem is $(1) \Rightarrow (3)$.
In the above theorem, the assertion does not hold in general if we drop the assumption that $e(A) > 1$ and $e(B)>1$; even if $R$ is almost Gorenstein and $r(R)\le r(A)r(B)$, $B$ is not necessarily Gorenstein when $e(A)=1$, that is, in the previous theorem, (1) does not imply (3).
However, it can be derived that $B$ is almost Gorenstein as follows:

Suppose that $e(A)=1$. 
Then we can rewrite (\ref{multiplicitycondition1}) as $\delta(B)\le \delta(R)$.
Since $R$ is almost Gorenstein and $r(R)\le r(B)$, we have
\begin{align}\label{rem_eq}
\delta(B)\le \delta(R) =r(R)-1\le r(B)-1 \le \delta(B).
\end{align}
Therefore, we get $\delta(B)=r(B)-1$, and hence $B$ is almost Gorenstein.

(b)
If we drop the assumption $r(R) \leq r(A)r(B)$ from (1), this does not imply (3) in general.
In fact, put $R=\QQ[s,st^{18},st^{21},st^{23},st^{26}]$, $A=\QQ[s,st,st^2]$ and $B=\QQ[s,st^9,st^{10},st^{13}]$.
By using $\mathtt{Macaulay2}$ ({\cite{M2}}), we can check $h_R(t)=1+3t+5t^2+7t^3+6t^4+3t^5+t^6$,
$h_{A}(t)=1+t$ and $h_{B}(t)=1+2t+3t^2+4t^3+2t^4+t^5$.
Moreover, $R$ is almost Gorenstein and
the equality of (\ref{multiplicitycondition1}) holds.
However, $B$ is not Gorenstein and $r(R)=3>2=r(A)r(B)$.
\end{rem}




\subsection{Sufficient conditions to satisfy the multiplicity condition}

In this subsection, we provide
sufficient conditions to satisfy the equality of (\ref{multiplicitycondition1})
for semi-standard graded rings.

Let $h=h(t)=\sum_{i=0}^sh_it^i \in \ZZ[t]$
and put $e(h):=\sum_{i=0}^sh_i=h(1)$, $\delta(h):=\sum_{i=0}^s(2i-s)h_i$.

\begin{rem}
Let $R$ be a Cohen--Macaulay semi-standard graded ring satisfies $(*)$
and let $h_R(t)$ be its $h$-polynomial.
Then we have
$e(R)=e(h_R)$ and $\delta(R)=\delta(h_R)$.
\end{rem}

\begin{prop}\label{ecformula}
Let $h,g \in \ZZ[t]$,
we have $\delta(hg)=e(h)\delta(g)+e(g)\delta(h)$.
In particular, if $g=1+t+\cdots+t^{a-1}$ for some integer $a>0$,
we have $\delta(hg)=a\delta(h)$.
\begin{proof}
Put $h(t)=a_0+a_1t+\cdots+a_nt^n$, $g(t)=b_0+b_1t+\cdots+b_mt^m$.
Then we have
$$\delta(hg)=\sum_{i=0}^{m+n} \left( (2i-(m+n)) \sum_{j=0}^i a_{i-j}b_j  \right)
=\sum_{j=0}^m \left(\sum_{i=0}^{m+n}(2i-(m+n))a_{i-j} \right)b_j.$$
Now we calculate the coefficient $A_j$ of $b_j$ in $\delta(hg)$ for each $1 \leq j \leq m$. 
Note that $a_i=0$ if $i<0$ or $n<i$.
Then we have
\begin{align*}
A_j=\sum_{i=j}^{j+n}(2i-(m+n))a_{i-j}&=\sum_{i=0}^{n}(2(i+j)-(m+n))a_{i} \\
&=\sum_{i=0}^n ((2i-n)a_i+(2j-m)a_i)
=\delta(h)+(2j-m)e(h).
\end{align*}
Therefore,
$$\delta(hg)=\sum_{i=0}^mA_ib_i=e(g)\delta(h)+e(h)\sum_{i=0}^m(2i-m)b_i=e(g)\delta(h)+e(h)\delta(g).$$
\end{proof}
\end{prop}

\begin{cor}\label{h_PRODUCTcondition}
Let $A$, $B$ and $R$ be Cohen--Macaulay semi-standard graded rings
satisfy $(*)$.
If $h_R(t)=h_{A}(t)h_{B}(t)$,
we have
$\delta(R)=e(B)\delta({A})+e(A)\delta({B})$.
\end{cor}

\begin{ex}
$A$, $B$ and $R$ in Example \ref{drop_assumption} (b) satisfy $h_{A}(t)h_{B}(t)=h_R(t)$.
Even by using Corollary \ref{h_PRODUCTcondition}, it can be confirmed that
$\delta(R)=e(B)\delta({A})+e(A)\delta({B})$.
\end{ex}

Now we apply Theorem \ref{MMMMM} to the tensor product of semi-standard graded rings and its quotient rings divided by a homogeneous regular sequence.

\begin{cor}\label{cor_tensor}
Let $A$ and $B$ be semi-standard graded Cohen--Macaulay $\kk$-algebras over a field $\kk$
such that $A$ and $B$ satisfy $(*)$ and $e(A) > 1$ and $e(B) > 1$.
In addition, let $T=A\otimes_\kk B$ and
let $\xb=x_1,\ldots,x_n$
be a homogeneous regular sequence on $T$
 with $\deg(x_k)=a_k$ for $1 \le k \le n$
such that $T/(\xb)$
satisfies $(*)$.
Set $R=T$ or $R=T/(\xb)$,
then the following conditions are equivalent:
\begin{itemize}
\item[(1)] $R$ is almost Gorenstein;
\item[(2)] $R$ is Gorenstein;
\item[(3)] $A$ and $B$ are Gorenstein.
\end{itemize}
\end{cor}

\begin{proof}
Note that $r(R)=r(T)=r(A)r(B)$ (cf. \cite[Theorem 3.3.5 (a)]{bruns1998cohen} and \cite[Theorem 4.2]{herzog2019trace}) and $R$ satisfies $(*)$ since $A$ and $B$ do.
First, we consider the case of $R=T$.
Since $h_{A}(t)h_{B}(t)=h_T(t)$,
the assertion follows from by Theorem \ref{MMMMM} and Corollary \ref{h_PRODUCTcondition}.
Next, we consider the case of $R=T/(\xb)$.
In this case, we can check
$h_R(t)=\left( \prod_{i=1}^n(1+t+\cdots+t^{a_i-1}) \right)h_T(t)$
so we have $\delta(R)=a_1\cdots a_n \delta(T)$
by Proposition \ref{ecformula}.
Then we have
$$\delta(R) \geq \delta(T)=e(B)\delta({A})+e(A)\delta({B})$$
by Corollary \ref{h_PRODUCTcondition}
so the assertion follows from Theorem \ref{MMMMM}.
\end{proof}

We can study polynomial extensions of almost Gorenstein rings by using our main theorem.
The semi-standard graded case can be proven easily.

\begin{cor}\label{polynomialextension}
Let $R$ and $S=R[x]$ be Cohen--Macaulay semi-standard graded rings over a field $\kk$.
Then the following conditions are equivalent:
\begin{itemize}
\item[(1)] $S$ is almost Gorenstein;
\item[(2)] $R$ is almost Gorenstein.
\end{itemize}
\begin{proof}
Note that $S=R\otimes_\kk \kk[x]$ and that $R$ satisfies $(*)$ if and only if $S$ does.
We have $r(R)=r(S)$ and $\delta({R}) = \delta({S})$
by Corollary \ref{h_PRODUCTcondition}.
Thus we can see that all the equalities appearing in (\ref{rem_eq}) hold, so
the assertion follows by the observation of Remark~\ref{drop_assumption} (a).
\end{proof}
\end{cor}

\section{Applications to affine semigroup rings}
In this section, we provide some applications of our results to affine semigroup rings; numerical semigroup rings, edge rings and stable set rings.

\subsection{Application to gluings of numerical semigroups}
In this section, we give another proof of Theorem~\ref{NUMEE} from viewpoints of our main theorem.

\begin{prop}\label{ORENOFORMULA}
Let $H_1$ and $H_2$ be numerical semigroups
and let ${x_1} \in H_1 \setminus G(H_1)$, $x_2 \in H_2 \setminus G(H_2)$ with $\gcd(x_1,x_2)=1$.
Moreover, let
$H=\langle{x_2}H_1,{x_1}H_2 \rangle$ be the gluing of $H_1$ and $H_2$.
We put $R=\kk[H]$
and $A_i=\kk[H_i]$
for $i=1,2$.
Then we have $$\delta(R)={x_2}\delta({A_1})+{x_1}\delta({A_2}).$$
\begin{proof}
For $i=1,2$, the Hilbert series of $A_i$ and $\omega_{A_i}(-a_i)$ are of the following forms
by Proposition \ref{canonicalHilbert} and \ref{aperyHilbert};
$$\H(A_i,t)= \frac{1}{1-t^{{x_i}}}\sum_{a \in \Ap(H_i,{x_i})} t^a,
\;\;\;\H({\omega_{A_i}}(-a_{A_i}),t)
=\frac{t^{F(H_i)+{x_i}}}{1-t^{x_i}}\sum_{a \in \Ap(H_i,{x_i})} \frac{1}{t^a}.
$$
For $i=1,2$, we can take graded $R$-monomorphisms $\phi: R \hookrightarrow \omega_R(-a_R)$, $\phi_i: A_i \hookrightarrow \omega_{A_i}(-a_{A_i})$
of degree 0 since $R$, $A_1$ and $A_2$ are domain.
Put
$C_{R}=\cok(\phi)$, $C_{A_i}=\cok(\phi_i)$ and
$f_i(t):=\H(C_{A_i},t)$ for $i=1,2$, then we have $f_i(t) \in \ZZ_{\ge 0}[t]$ and
$\delta({A_i})=f_i(1)$
since $\dim C_{A_i}=0$.
Now we put $g(t):=1-t^{x_1x_2}$, for $(p,q)=(1,2),(2,1)$, we have
$$(t^{x_q})^{F(H_p)+{x_p}}\sum_{a\in \Ap(H_p,{x_p})}
\frac{1}{(t^{x_q})^{a}}=g(t)f_{p}(t^{x_q})+\sum_{a\in \Ap(H_p,{x_p})}(t^{x_q})^{a}$$
since $f_i(t)=\H(\omega_{A_i}(-a_{A_i}),t)-\H(R,t)$ for $i=1,2$.
Put $F_1(t):=\sum_{a\in \Ap(H_1,{x_1})}
(t^{{x_2}})^a$, $F_2(t):=\sum_{b\in \Ap(H_2,{x_2})} (t^{{x_1}})^b$.
Note that $F_i(1)=x_i$ for $i=1,2$.
By using lemma \ref{Nari1}, we can check
the following.
$$
\H(R,t)
= \frac{1}{g(t)}\sum_{w \in \Ap(H,{x_1x_2})} t^w
= \frac{1}{g(t)}\sum_{
\substack{ a\in \Ap(H_1,{x_1}) \\ b\in \Ap(H_2,{x_2}) }
} t^{a{x_2}+b{x_1}}
= \frac{F_1(t)F_2(t)}{g(t)},
$$
\begin{equation}\notag
\begin{split}
\H(\omega_{R}(-a_R),t)
&=\frac{1}{g(t)}
\left( (t^{x_2})^{F(H_1)+{x_1}}\sum_{a\in \Ap(H_1,{x_1})}
\frac{1}{(t^{x_2})^{a}} \right) \left( (t^{x_1})^{F(H_2)+{x_2}} \sum_{b\in \Ap(H_2,{x_2})} \frac{1}{(t^{x_1})^{b}}\right)\\
&=\frac{(
g(t)f_1(t^{x_2})+F_1(t)
)
(g(t)f_2(t^{x_1})+F_2(t))}{g(t)}\\
&=g(t)f_1(t^{x_2})f_2(t^{x_1})+
F_1(t)f_2(t^{x_1})
+
F_2(t)f_1(t^{x_2})+\H(R,t),
\end{split}
\end{equation}

$$
\H(C_R,t)
= \H({\omega_R}(-a_R),t)-\H(R,t)
=
g(t)f_1(t^{x_2})f_2(t^{x_1})+F_1(t)f_2(t^{x_1})
+F_2(t)f_1(t^{x_2}).
$$

Therefore, we have
$\delta(R)=\H(C_R,1)=
F_2(1)f_1(1)+F_1(1)f_2(1)
={x_2}\delta({A_1})+{x_1}\delta({A_2})$.
\end{proof}
\end{prop}

\begin{thm}[{\cite[Theorem 6.7]{nari2013symmetries}} and \cite{De}]
\label{NUMEGLUINGOYO}
Let $H_1$ and $H_2$ be numerical semigroups
and let ${x_1} \in H_1 \setminus G(H_1)$, $x_2 \in H_2 \setminus G(H_2)$ with $\gcd(x_1,x_2)=1$.
Moreover, let
$H=\langle{x_2}H_1,{x_1}H_2 \rangle$ be the gluing of $H_1$ and $H_2$.
Then the following conditions are equivalent:
\begin{itemize}
\item[(1)] $\kk[H]$ is almost Gorenstein;
\item[(2)] $\kk[H]$ is Gorenstein;
\item[(3)] $\kk[H_1]$ and $\kk[H_2]$ are Gorenstein.
\end{itemize}
\begin{proof}
Since
${x_2} > e(\kk[H_2])$, ${x_1} > e(\kk[H_1])$,
we have
$$\delta({\kk[H]}) \geq e(\kk[H_2])\delta({\kk[H_1]})+e(\kk[H_1])\delta({\kk[H_2]})$$
by Proposition \ref{ORENOFORMULA}.
Thus the assertion follows from Lemma \ref{Nari1} and Theorem \ref{MMMMM}.
\end{proof}
\end{thm}


\subsection{Application to edge rings and stable set rings}
In this subsection, we prove Theorems~\ref{main_edge} and \ref{main_stab}.
Before that, we discuss toric splittings of toric ideals, which helps us to show our results.

\textit{Toric ideals} are prime ideals of polynomial rings generated by binomials.
We say that a toric ideal $I$ of a polynomial ring $S$ over $\kk$ has \textit{toric splitting} (or $I$ is a \textit{splittable toric ideal}) if there exist toric ideals $I_1$ and $I_2$ of $S$ such that $I=I_1+I_2$.
See \cite{favacchio2021splittings} for details on toric splittings.

For a polynomial $f$ of $S$, we denote the set of variables of $S$ appearing in $f$ by $\var(f)$ and for a subset $F$ of $S$, we define $\var(F):=\bigcup_{f\in F} \var(f)$.

\begin{prop}\label{prop_split}
Let $I$ be a toric ideal of $S$ generated by binomials $f_1,\ldots,f_n,g_1,\ldots,g_m$ and let $B_1=\{f_1,\ldots,f_n\}$ and $B_2=\{g_1,\ldots,g_m\}$.
Suppose that $I$ has the toric splitting $I=(B_1)+(B_2)$, that $\var(B_1\cup B_2)$ contains all variables of $S$ and that $\var(B_1)\cap \var(B_2)$ consists a single variable $z$.
Then we have
$$S/I\cong \biggl( \kk[\var(\widetilde{B}_1)]/(\widetilde{B}_1)\otimes_\kk \kk[\var(\widetilde{B}_2)]/(\widetilde{B}_2)\biggr) /(z_1-z_2),$$
where $z_1$ and $z_2$ are new variables that do not belong to $S$ and we let $\tilde{f}_i$ (resp. $\tilde{g}_i$) be the polynomial obtained by substituting $z_1$ (resp. $z_2$) for $z$ appearing in $f_i$ (resp. $g_i$) and let $\widetilde{B}_1=\{\tilde{f}_1,\ldots,\tilde{f}_n\}$ and $\widetilde{B}_2=\{\tilde{g}_1,\ldots,\tilde{g}_m\}$.
\end{prop}

\begin{proof}
Since $\var(B_1\cup B_2)$ contains all variables of $S$ and $I$ has the toric splitting $I=(B_1)+(B_2)$, we have $\kk[\var(B_1\cup B_2)]/((B_1)+(B_2))=S/I$.

Moreover, it follows from
$\var(\widetilde{B}_1)\cap \var(\widetilde{B}_2)=\emptyset$ that
$$\kk[\var(\widetilde{B}_1)]/(\widetilde{B}_1)\otimes_\kk \kk[\var(\widetilde{B}_2)]/(\widetilde{B}_2)\cong \kk[\var(\widetilde{B}_1\cup \widetilde{B}_2)]/((\widetilde{B}_1)+(\widetilde{B}_2)),$$
and hence
\begin{align*}
S/I&=\kk[\var(B_1\cup B_2)]/((B_1)+(B_2)) \\
&\cong \kk[\var(\widetilde{B}_1\cup \widetilde{B}_2)]/((\widetilde{B}_1)+(\widetilde{B}_2)+(z_1-z_2)) \\
&\cong \biggl( \kk[\var(\widetilde{B}_1)]/(\widetilde{B}_1)\otimes_\kk \kk[\var(\widetilde{B}_2)]/(\widetilde{B}_2)\biggr) /(z_1-z_2).
\end{align*}
\end{proof}

\begin{cor}\label{cor_toric_split}
Work with the notation and hypothesis of Proposition~\ref{prop_split}, and assume that $A:=\kk[\var(\widetilde{B}_1)]/(\widetilde{B}_1)$, $B:=\kk[\var(\widetilde{B}_2)]/(\widetilde{B}_2)$ and $R:=S/I$ are semi-standard graded.
Then the following conditions are equivalent:
\begin{itemize}
\item[(1)] $R$ is almost Gorenstein;
\item[(2)] $R$ is Gorenstein;
\item[(3)] $A$ and $B$ are Gorenstein.
\end{itemize}
\end{cor}

\begin{proof}
It follows from Corollary~\ref{cor_tensor} and Proposition~\ref{prop_split} (even if $\var(B_1\cup B_2)$ does not necessarily contain all variables of $S$, we can show this claim by applying polynomial extension and Corollary~\ref{polynomialextension}).
\end{proof}

We now give the proofs of our theorems and some examples.
First, we provide an application to edge rings:

\begin{thm}\label{thm_ed}
Let $G_1$ and $G_2$ be simple graphs.
Suppose that $G_1$ is bipartite and $\kk[G_2]$ is Cohen--Macaulay with $e(\kk[G_i])>1$ for $i=1,2$.
Then the following are equivalent: 
\begin{itemize}
\item[(1)] $\kk[G_1\sharp G_2]$ is almost Gorenstein;
\item[(2)] $\kk[G_1\sharp G_2]$ is Gorenstein;
\item[(3)] $\kk[G_1]$ and $\kk[G_2]$ are Gorenstein.
\end{itemize}
\end{thm}

\begin{proof}
It is known that the edge ring of a bipartite graph is normal (\cite[Corollary~2.3]{ohsugi1998normal}), and hence $\kk[G_1]$ is Cohen--Macaulay.
Note that $k:=|V(G_1)\cap V(G_2)|\le 2$ since $G_1$ is bipartite.
If $k\le 1$, then we have $\kk[G_1\sharp G_2]\cong \kk[G_1]\otimes_{\kk}\kk[G_2]$ (cf. \cite[Proposition~10.1.48]{villarreal2018monomial}).
Therefore, conditions (1), (2) and (3) are equivalent by Corollary~\ref{cor_tensor}.

Suppose that $k=2$ and let $c=V(G_1)\cap V(G_2)$ (note that $c\in E(G)$).
We can see that $I_G$ has the toric splitting $I_G=I_{G_1}+I_{G_2}$ (\cite[Corollary~4.8]{favacchio2021splittings}) and the common variable that appears in generators of $I_{G_1}$ and $I_{G_2}$ is only $x_c$.
Therefore, we get the desired result from Corollary~\ref{cor_toric_split}.
\end{proof}


Next, we present an application to stable set rings:

\begin{thm}\label{thm_st}
Let $G_1$ and $G_2$ be simple graphs.
Suppose that $\Stab_\kk(G_i)$ is Cohen--Macaulay with $e(\Stab_\kk(G_i))>1$ for $i=1,2$.
Then the following are equivalent:
\begin{itemize}
\item[(1)] $\Stab_\kk(G_1 + G_2)$ is almost Gorenstein;
\item[(2)] $\Stab_\kk(G_1 + G_2)$ is Gorenstein;
\item[(3)] $\Stab_\kk(G_1)$ and $\Stab_\kk(G_2)$ are Gorenstein.
\end{itemize}
\end{thm}

\begin{proof}
It follows immediately from Corollary~\ref{cor_toric_split} and Lemma~\ref{lem_stab_spli} below.
\end{proof}

\begin{lem}\label{lem_stab_spli}
Let $G_1$ be $G_2$ two graphs with $V(G_1)\cap V(G_2)=\emptyset$ and 
let $B_1=\{f_1,\ldots,f_n\}$ and $B_2=\{g_1,\ldots,g_m\}$ be minimal generating systems of $J_{G_1}$ and $J_{G_2}$, respectively.
Then we have $\var(B_1)\cap \var(B_2)=\{x_{\emptyset}\}$.
Moreover, $J_{G_1+G_2}$ has the toric splitting $J_{G_1+G_2}=J_{G_1}+J_{G_2}$. 
\end{lem}

\begin{proof}
Notice that each non-empty stable set of $G_1+G_2$ is that of only one of $G_1$ or $G_2$.
The first assertion holds from this fact.

We can see that $J_{G_1}+J_{G_2}\subset J_G$ since $J_{G_i}\subset J_G$ for $i=1,2$.
Thus it is enough to show that any binomial $f\in J_{G_1+G_2}$ belongs to $J_{G_1}+J_{G_2}$.
We may assume that $f$ can be written as
$$f=x_{S_1}\cdots x_{S_p}x_{T_1}\cdots x_{T_q} -
x_{S'_1}\cdots x_{S'_r}x_{T'_1}\cdots x_{T'_u}x_{\emptyset}^a,$$
where $S_i$ and $S'_j$ (resp. $T_k$ and $T'_l$) are non-empty stable sets of $G_1$ (resp. $G_2$), and $a=p+q-r-u\ge 0$.
Moreover, we have $(S_i\cup S'_j)\cap (T_k\cup T'_l)=\emptyset$ for each $i,j,k$ and $l$, which implies that
$x_{S_1}\cdots x_{S_p}x_{\emptyset}^b-x_{S'_1}\cdots x_{S'_r}x_{\emptyset}^{b'}$
and 
$x_{T_1}\cdots x_{T_q}x_{\emptyset}^c-x_{T'_1}\cdots x_{T'_u}x_{\emptyset}^{c'}$
belong to $J_{G_1}$ and $J_{G_2}$ for some $b,b',c$ and $c'$, respectively.

Suppose that $a_1=p-r\ge 0$.
Then we can see that
\begin{align*}
f=\begin{cases}
x_{T_1}\cdots x_{T_q}(x_{S_1}\cdots x_{S_p}-x_{S'_1}\cdots x_{S'_r}x_{\emptyset}^{a_1})+ &\; \\
\hspace{3.0cm} x_{S'_1}\cdots x_{S'_r}x_{\emptyset}^{a_1}(x_{T_1}\cdots x_{T_q}-x_{T'_1}\cdots x_{T'_u}x_{\emptyset}^{a-a_1})  &\text{ if } a\ge a_1, \\
\vspace{-0.4cm}\; &\; \\
x_{T_1}\cdots x_{T_q}(x_{S_1}\cdots x_{S_p}-x_{S'_1}\cdots x_{S'_r}x_{\emptyset}^{a_1})+ &\; \\
\hspace{3.0cm} x_{S'_1}\cdots x_{S'_r}x_{\emptyset}^{a}(x_{T_1}\cdots x_{T_q}x_{\emptyset}^{a_1-a}-x_{T'_1}\cdots x_{T'_u})  &\text{ if } a\le a_1,
\end{cases}
\end{align*}
and hence $f\in J_{G_1}+J_{G_2}$.

If $a_1<0$, we have $a_2=q-u\ge0$.
Otherwise, $a=(p-r)+(q-u)=a_1+a_2<0$, contradicting the fact that $a\ge 0$.
Therefore, we can show that $f\in J_{G_1}+J_{G_2}$ by the same argument and conclude that $J_{G_1+G_2}=J_{G_1}+J_{G_2}$.

\end{proof}

Finally, we give two specific applications of our theorems.

\begin{ex}
(a) For two integers $m\ge 1$ and $l\ge 0$, let $\fkg^l_{r_1,\ldots,r_m}$ be the graph consisting of $r_j$ cycles of length $2j + 1$ and $l$ even cycles, such that all cycles share a single common vertex.
If $l=0$, then $\fkg^l_{r_1,\ldots,r_m}$ consists of only odd cycles and we denote it by $\fkg_{r_1,\ldots,r_m}$.
This graph is the clique sum of $\fkg_{r_1,\ldots,r_m}$ and $l$ even cycles.
It is known that the edge ring of an even cycle $C$ is Gorenstein with $e(\kk[C])>1$ (cf. \cite[Theorem~2.1(b)]{ohsugi2006special}) and that almost Gorensteinness of the edge ring of $\fkg_{r_1,\ldots,r_m}$ has been investigated in \cite[Theorem~1.2]{bhaskara2023h}.

From these facts and Theorem~\ref{thm_ed}, we can characterize when $\kk[\fkg^l_{r_1,\ldots,r_m}]$ is (almost) Gorenstein:
\begin{itemize}
\item $\kk[\fkg^l_{r_1,\ldots,r_m}]$ is Gorenstein if and only if $r_1+\cdots +r_m\le 2$.
\item $\kk[\fkg^l_{r_1,\ldots,r_m}]$ is not Gorenstein but alomst Gorenstein if and only if $l=0$, $m=1$ and $r_1\ge 3$.
\end{itemize}


(b) Let $G_1,\ldots,G_k$ be simple connected graphs with at most one $G_i$ not being bipartite and let $C$ be an even cycle with at least $k$ edges.
In addition, let $G$ be the graph obtained by identifying an edge of $G_i$ with a distinct edge of $C$ for each $i$ (this graph appears in \cite[Theorem~3.7]{favacchio2021splittings}).
If $e(\kk[G_i])>1$ for all $i$, then we have the following equivalent conditions:
\begin{align*}
\kk[G] \text{ is almost Gorenstein} \Leftrightarrow
\; \kk[G] \text{ is Gorenstein} \Leftrightarrow
\; \kk[G_i] \text{ is Gorenstein for all $i$}.
\end{align*}
\end{ex}

\begin{ex}
For $m\ge 3$ and $n\ge 1$, we consider the \textit{cone graph} $C_m+\overline{K}_n$, which is the join of the cycle graph $C_m$ of length $m$ and the empty graph $\overline{K}_n$ of order $n$.
It is known that
\begin{itemize}
\item the stable set ring of $C_m$ is almost Gorenstein (\cite[Theorem~4.1]{miyazaki2023non}). In particular, it is Gorenstein if and only if $m$ is even, $m=3$ or $5$ (\cite[Theorem~2.1(b)]{ohsugi2006special} and \cite[Theorem~1]{hibi2019odd}).
\item the stable set ring of $\overline{K}_n$ is isomorphic to the Segre product of $n$ polynomial rings in $2$ variables and is Gorenstein.
Note that $e(\Stab_\kk(\overline{K}_n))=1$ if and only if $n=1$.
\end{itemize}
Therefore, it follows from Theorem~\ref{thm_st} that
\begin{itemize}
\item $\Stab_\kk(C_m+\overline{K}_n)$ is Gorenstein if and only if $m$ is even, $m=3$ or $5$.
\item $\Stab_\kk(C_m+\overline{K}_n)$ is not Gorenstein but almost Gorenstein if and only if $n=1$, $m$ is odd and $m\ge 7$.
\end{itemize}
\end{ex}



\end{document}